\documentclass[a4paper,11pt]{amsart}

\usepackage[english]{babel} %

  \usepackage{lmodern}
  \usepackage[T1]{fontenc}

\usepackage{amssymb}
\usepackage{mathtools} %
\usepackage{slashed} %
\usepackage[babel=true]{microtype} %
\usepackage[autostyle=true]{csquotes} %
\usepackage[dvipsnames]{xcolor} %
\usepackage[biblatex=true]{embrac} %

\calclayout

\definecolor{Heather}{RGB}{164, 132, 172}
\usepackage[pdfusetitle,bookmarks,pdfpagelabels]{hyperref} %
\hypersetup{
  colorlinks,
  urlcolor=Periwinkle,
  citecolor=Heather,
  linkcolor=teal,
  breaklinks=true,
  final,
}

\usepackage[
    citestyle=numeric-comp,
    bibstyle=numeric-comp,
    backend=biber,
    url=true,
    doi=true,
    isbn=false,
    giveninits=true,
    maxbibnames=100,
    maxcitenames=10,
    sorting=nyt
]{biblatex}

\DeclareDelimFormat[textcite]{multinamedelim}{\textendash}
\DeclareDelimAlias[textcite]{finalnamedelim}{multinamedelim}

\DeclareSourcemap{
  \maps{
    \map{ %
        \step[fieldset = month, null]
        \step[fieldset = day, null]
        \step[fieldset = language, null]
    }
    \map{ %
        \step[fieldsource = doi, final]
        \step[fieldset = url, null]
    }
    \map{ %
        \step[fieldsource = eprint, final]
        \step[fieldset = url, null]
    }
  }
}

\usepackage{tikz}
\usetikzlibrary{cd} %

\usepackage[shortlabels,inline]{enumitem}
\setlist[enumerate,1]{label=\upshape(\arabic*)}
\newlist{myenumi}{enumerate}{1}
\setlist[myenumi,1]{label=\upshape(\roman*)}
\newlist{myenuma}{enumerate}{1}
\setlist[myenuma,1]{label=\upshape(\alph*)}

\usepackage{thmtools}
\declaretheorem[name=Theorem, numberwithin=section]{theorem}
\declaretheorem[name=Theorem, numbered=no]{theorem*}
\declaretheorem[name=Lemma, numberlike=theorem]{lemma}
\declaretheorem[name=Lemma, numbered=no]{lemma*}
\declaretheorem[name=Corollary, numberlike=theorem]{corollary}

\declaretheorem[name=Remark, numberlike=theorem, style=remark]{remark}

\declaretheorem[name=Theorem]{theoremx}

\numberwithin{equation}{section}
\allowdisplaybreaks[1]

\usepackage[noabbrev,capitalise]{cleveref} %
\crefdefaultlabelformat{#2\textup{#1}#3}

\usepackage[disable]{todonotes}
\NewDocumentEnvironment{myTodo}{m}{\hypersetup{hidelinks}\textsf{\textbf{#1:}} }{}

\makeatletter
\providecommand\@dotsep{5}
\def\listtodoname{List of Todos}
\def\listoftodos{\@starttoc{tdo}\listtodoname}
\makeatother

\usepackage{mymacros}
\NewDocumentCommand{\CP}{}{\mathbb{CP}}
\NewDocumentCommand{\Spinc}{}{\Spin^c}
\NewDocumentCommand{\torsion}{}{\operatorname{torsion}}

\DeclareMathOperator{\sys}{sys}

\newcommand{\st}{\textup{st}}
\newcommand{\FS}{\textup{FS}}
\newcommand{\Td}{\operatorname{Td}}

\NewDocumentCommand{\pair}{m m}{\innp{#1}{#2}}
\NewDocumentCommand{\comassp}{m m}{\abs{#1}_{\ast,#2}}
\NewDocumentCommand{\Comass}{m}{\norm{#1}_{\ast}}
\NewDocumentCommand{\stnorm}{m}{\norm{#1}_{\mathrm{st}}}
\NewDocumentCommand{\comassnorm}{m}{\norm{#1}_{\mathrm{comass}}}

\title{Stable \texorpdfstring{\(2\)}{2}-systoles, scalar curvature and spin\texorpdfstring{\(^c\)}{-c} comass bounds}

\author{Simone Cecchini}
\address[Simone Cecchini]{Texas A\&M University, Department of Mathematics, 155 Ireland Street, College Station, TX 77843-3368, USA}
\email{\href{mailto:cecchini@tamu.edu}{cecchini@tamu.edu}}
\urladdr{\href{https://simonececchini.org}{simonececchini.org}}

\author{Sven Hirsch}
\address[Sven Hirsch]{Columbia University, 2990 Broadway, New York, NY 10027, USA}
\email{\href{mailto:sven.hirsch@columbia.edu}{sven.hirsch@columbia.edu}}
\urladdr{\href{https://www.svenhirsch.com}{www.svenhirsch.com}}

\author{Rudolf Zeidler}
\thanks{SC: Supported by a grant from the Simons Foundation (MPS-TSM-00007902, SC)}
\thanks{RZ: Funded by the European Union (ERC Starting Grant 101116001 – COMSCAL). Views and opinions expressed are however those of the author(s) only and do not necessarily reflect those of the European Union or the European Research Council. Neither the European Union nor the granting authority can be held responsible for them.}
\thanks{RZ: Funded by the Deutsche Forschungsgemeinschaft (DFG, German Research Foundation) – Project numbers
523079177; %
427320536; %
390685587%
}
\address[Rudolf Zeidler]{Universität Potsdam, Institut für Mathematik, Karl-Liebknecht-Str.\ 24--25, 14476 Potsdam, Germany}
\email{\href{mailto:rudolf.zeidler@uni-potsdam.de}{rudolf.zeidler@uni-potsdam.de}}
\urladdr{\href{https://www.rzeidler.eu}{www.rzeidler.eu}}

\hypersetup{
  pdfauthor={Simone Cecchini, Sven Hirsch, Rudolf Zeidler}, %
}
\begin{document}

\begin{abstract}
We prove a sharp stable \(2\)-systolic inequality for complex projective space under the scalar curvature lower bound of the normalized Fubini--Study metric. If \(M\) is diffeomorphic to \(\CP^n\) and \(\scal_g\ge 4n(n+1)\), then \(\sys_2^\st(M,g)\le \pi\).
Moreover, equality holds only for the Fubini--Study metric, up to biholomorphism after choosing the corresponding complex structure. The proof uses \(\Spinc\) Dirac operators, a comass estimate for the curvature term in the Lichnerowicz formula, and stable norm--comass duality.
\end{abstract}

\maketitle

\section{Introduction}

Lower scalar curvature bounds often lead to rigidity and extremality results that
single out canonical model geometries. Classical examples include the Geroch
conjecture \cite{SchoenYau1, GromovLawson} for flat tori, the positive mass theorem for Euclidean \cite{SchoenYau2, Witten} and hyperbolic
space \cite{Wang, ChruscielHerzlich}, and Llarull's theorem \cite{Llarull} for the round sphere.

\medskip

In this paper we study the relationship between lower scalar curvature bounds and the stable 2-systole.
Stable systoles are homological systolic invariants in which one measures the asymptotic size of nontrivial homology classes. More
precisely, if \(h\in \HZ_2(M;\Z)/\torsion\), its stable norm is the asymptotic minimal area per multiple of \(h\):
\[
   \stnorm{h}
   :=
   \lim_{k\to\infty}
   \frac{1}{k}
   \inf\{\Area(T): T \text{ is an integral }2\text{-cycle representing } kh\}.
\]
The stable 2-systole is then defined by
\[
   \sys^\st_2(M,g)
   :=
   \min\{\stnorm{h}:
          0\neq h\in \HZ_2(M;\Z)/\torsion\}.
\]
This invariant is particularly well suited to scalar curvature questions in
degree two. On the one hand, it records the smallest stabilized area of a
nontrivial two-dimensional homology class. On the other hand, by the
Federer--Gromov duality theorem, the stable norm on homology is dual to the
comass norm on cohomology.
Since the curvature term in the Lichnerowicz formula for the Dirac operator twisted with a line bundle is controlled by the comass of a
degree-two curvature form, this duality turns an analytic scalar curvature estimate into a stable systolic inequality.
This point of view was introduced by Gromov~\cite[Section~5.9]{Gromov_four_lectures}, who proved a sharp systolic inequality with scalar curvature on 3-manifolds.
More recently, these ideas have been developed further by Orikasa~\cite{orikasa2025systolicinequalityscalarcurvature} and Stryker~\cite{Stryker2026Stable2Systole}.

\medskip

The main result of this paper is a sharp upper bound for the stable 2-systole of complex projective space \(\CP^n\), together with a rigidity statement for the equality case.
Let \(J_0\) be the standard complex structure on \(\CP^n\).
We normalize the Fubini--Study metric \(g_\FS\) on \((\CP^n,J_0)\) by
\[
   \int_\ell \omega_\FS=\pi
\]
for every projective line \(\ell\subset \CP^n\). With this normalization,
\[
   \scal_{g_\FS}=4n(n+1),
\]
and every projective line has area \(\pi\).
Our main theorem establishes extremality and rigidity for the stable 2-systole of Riemannian metrics on \(\CP^n\) with scalar curvature bounded from below by \(4n(n+1)\).
This characterizes the Fubini--Study metric as the unique Riemannian metric, up to biholomorphism, with constant scalar curvature \(4n(n+1)\) and stable 2-systole \(\pi\), in the precise sense specified by the following theorem.

\begin{theoremx}\label{thm:main}
Let \(M\) be a smooth manifold diffeomorphic to \(\CP^n\).
Let \(g\) be a Riemannian metric on \(M\) such that \(\scal_g \ge 4n(n+1)\).
Then
\[
\sys_2^\st(M,g)\le \pi.
\]
Moreover, if equality holds, then there exist a complex structure \(J\) on \(M\) and
a biholomorphism \(B \colon (M,J)\to (\CP^n,J_0)\) such that
\[
g=B^*g_\FS.
\]
\end{theoremx}

\begin{remark}
Stryker~\cite[Theorem~1.6]{Stryker2026Stable2Systole} recently obtained an upper bound for the stable 2-systole of \(\CP^{2q+1}\) with scalar curvature bounded from below by a positive constant.
These bounds are not sharp.
For example, on \(\CP^3\) with \(\scal_g \ge 48 = \scal_{g_\FS}\), Stryker's explicit estimate~\cite[Theorem~1.8]{Stryker2026Stable2Systole} gives \(\sys_2^\st(\CP^3,g) \le 5\pi\).
In contrast, Theorem~\ref{thm:main} gives the sharp upper bound for the stable \(2\)-systole of \(\CP^n\) in every complex dimension and shows that equality occurs only for the Fubini--Study metric, up to biholomorphism.
\end{remark}

\begin{remark}
In real dimension four, the case \(M\cong \CP^2\) also follows from LeBrun's perturbed Seiberg--Witten scalar curvature estimate~\cite[Theorem~1]{LeBrun:YamabeSW}.
Indeed, applying this estimate to the \(g\)-harmonic representative of the positive generator \(h\in \HZ^2(\CP^2;\Z)\), and using stable
norm--comass duality, gives
\[
\scal_g\ge 24
\qquad\Longrightarrow\qquad
\sys_2^\st(\CP^{2},g)\le \pi .
\]
The equality case follows from LeBrun's rigidity statement~\cite[Theorem~2]{LeBrun:YamabeSW}, together with uniqueness of Kähler--Einstein metrics on \(\CP^{2}\).
\end{remark}

\begin{remark}\label{rem:homological-systole}
Suppose that \(g\) is a Kähler metric on the standard complex manifold \((\CP^n,J_0)\) with \([\omega_g]=\lambda h\), where \(h=c_1(\mathcal O_{\CP^n}(1))\).
Then the Kähler form \(\omega_g\) calibrates projective lines and the stable 2-systole \(\sys_2^\st(\CP^n,g)\) coincides with the homological 2-systole
\[
\sys_2(\CP^n,g):=
\inf\bigl\{\Area_g(\Sigma)\bigr\},
\]
where \(\Sigma\) ranges over all integral \(2\)-cycles with \(0\neq [\Sigma]\in \HZ_2(\CP^n;\Z)\).
More precisely, every projective line \(\ell\simeq \CP^1\) is area minimizing in its real homology class and
\[
\sys^\st_2(\CP^n,g)=\sys_2(\CP^n,g)=\Area_g(\ell)=\int_\ell \omega_g=\lambda.
\]
Thus, for Kähler metrics on \((\CP^n,J_0)\), Theorem~\ref{thm:main} recovers the sharp homological 2-systole estimate
\[
\scal_g\ge 4n(n+1)
\quad\Longrightarrow\quad
\sys_2(\CP^n,g)\le \pi,
\]
with equality only if \(g\) is the Fubini--Study metric, up to a holomorphic automorphism of \((\CP^n;J_0)\).
In this Kähler setting the inequality itself also follows directly from the Chern--Weil formula for the total scalar curvature; the point of Theorem~\ref{thm:main} is that no Kähler assumption is imposed on \(g\).
For \(n=2\), this recovers Sha's sharp 2-systole upper bound and rigidity result for Kähler metrics on
\(\CP^2\); see \cite[Theorem 1.1]{sha20262systolecompactkahlersurfaces}.
\end{remark}

Thus Theorem~\ref{thm:main} shows that, among all Riemannian metrics on \(\CP^n\) with scalar curvature bounded below by that of the normalized Fubini--Study metric, the stable \(2\)-systole is maximized precisely by the Fubini--Study metric.
In this sense, Theorem~\ref{thm:main} is a scalar curvature rigidity theorem for complex projective space, in the same spirit as the rigidity results mentioned above.

\medskip

We also prove an odd-dimensional analogue:

\begin{theoremx}\label{thm:main-odd}
Let \(M\) be a smooth manifold diffeomorphic to \(\CP^n\times \Sphere^1\). Let \(g\) be a Riemannian metric on \(M\) such that \(\scal_g\geq 4n(n+1)\).
Then
\[
 \sys_2^\st(M,g)\leq \pi.
\]
Moreover, if equality holds, then the universal cover of \((M,g)\) is isometric to
\[
 (\CP^n,g_\FS)\times \R,
\]
where \(g_\FS\) is normalized as above.
\end{theoremx}

\begin{remark}
    When \(n=1\), Theorem~\ref{thm:main-odd} recovers \cite[Section~5.9]{Gromov_four_lectures}.
\end{remark}
Let us briefly indicate the mechanism behind these results.
Underlying Theorem \ref{thm:main} is the following general comass comparison result which may be of independent interest.

\begin{theoremx}\label{thm:general_comass_comparison intro}
    Let \(M^{2m}\) be an even-dimensional closed connected oriented manifold and let \(c \in \HZ^2(M;\Z)\) such that \(c \equiv w_2(\T M) \mod 2\) and
    \[\pair{[M]}{\eu^{c/2}\Ahat(\T M)} \neq 0.\]
    Let \(c_{\R}\in \HZ^2(M;\R)\) denote the image of \(c\) under the natural map \(\HZ^2(M;\Z)\to \HZ^2(M;\R)\).
    Then for every Riemannian metric \(g\) on \(M\), we have
    \[
        \min_{M} \scal_g \leq 4 m \pi \comassnorm{c_{\R}}.
    \]
    Moreover, if equality holds and \(c_{\R}\neq 0\), then \((M,g)\) is Kähler--Einstein with Kähler form \(\omega\) satisfying
    \(
        [\omega]=\frac{c_{\R}}{\comassnorm{c_{\R}}}
    \)
    in \(\HZ^2(M;\R)\).
\end{theoremx}

We also have an odd-dimensional analogue of Theorem~\ref{thm:general_comass_comparison intro}, see \cref{thm:general_comass_comparison_odd}, which is needed to prove Theorem~\ref{thm:main-odd}.

\begin{remark}
    This rigidity result is closely related to results of \textcite{moroianu-parallel} and \textcite[Theorem~1.9]{Goette-Semmelmann:spinC}, from which one would already expect that \(M\) must locally be a product of a Kähler manifold with a Ricci-flat manifold.
    The reason why there is no Ricci-flat factor in our case is that our formulation of the inequality in terms of the comass together with \(c_{\R}\neq 0\) forces equality in our sharp pointwise Clifford-algebra estimate (\cref{lem:general-clifford}), which yields a stronger rigidity conclusion.
    An additional technical issue that we have to overcome in our case is that the form minimizing the comass in a given cohomology class is a priori not necessarily smooth.
\end{remark}

The relevant \(\Spinc\) index condition in Theorem \ref{thm:general_comass_comparison intro} and \cref{thm:general_comass_comparison_odd} produces a harmonic spinor: directly in even dimensions and by spectral flow in odd dimensions.
The \(\Spinc\) Lichnerowicz formula then relates the scalar curvature to the curvature of the determinant line bundle.
A sharp pointwise Clifford-algebra estimate (\cref{lem:general-clifford}) bounds this curvature term in terms of the comass of the underlying degree-two cohomology class.
Finally, stable norm--comass duality (\cref{corollary:rank-one-duality}) converts the resulting scalar curvature estimate into the stable systolic inequalities above.

\medskip

The equality case is governed by the equality case in the Clifford estimate.
In even dimensions, this forces the relevant two-form to be a parallel Kähler form and the metric to be Kähler--Einstein.
In odd dimensions, the same argument produces a parallel two-form with one-dimensional kernel, which yields the isometric splitting of the universal cover.
For \(\CP^n\) and \(\CP^n\times \Sphere^1\), the final rigidity statements follow from the Hirzebruch--Kodaira--Yau and Kobayashi--Ochiai characterizations of projective space, together with the uniqueness of Kähler--Einstein metrics on \(\CP^n\).

\medskip

The paper is organized as follows.
In \cref{section:notation} we recall the background material on the comass norm and \(\Spinc\) structures.
\cref{sec:abstract-results} contains the main abstract estimates: we prove sharp scalar curvature--comass bounds for \(\Spinc\) manifolds, together with the corresponding rigidity statements in the equality case.
In \cref{sec:applications} we turn to geometric applications.
After recalling the stable norm--comass duality, we apply the results of \cref{sec:abstract-results} to prove Theorems~\ref{thm:main} and~\ref{thm:main-odd}.

\subsection*{Acknowledgments}
Part of this work was carried out at the Lonavala Geometry Festival and the authors would like to acknowledge the ideal working conditions.
SC and RZ gratefully acknowledge the hospitality of the Columbia Mathematics Department.
SC and SH thank Demetre Kazaras for several inspiring conversations on Gromov's spinorial approach to scalar-curvature systolic inequalities.
We thank Luca Di Cerbo for pointing out LeBrun's work on Yamabe constants~\cite{LeBrun:YamabeSW}, which in real dimension four recovers our sharp bound and rigidity.

\section{Notation and conventions}\label{section:notation}

Throughout, \(M\) denotes a closed oriented manifold and \(g\) a Riemannian metric on
\(M\). Since only degree two enters the sequel, we formulate the notions of mass,
comass, and stable norm in that degree.

\subsection{Comass}\label{ss:comass}

For a real \(2\)-form \(\omega \in \Dforms^2(M)\) and \(p \in M\), define the
pointwise comass of \(\omega\) at \(p\) by
\[
\comassp{\omega}{p}
:=
\sup \bigl\{
\abs{\omega_p(v)} : v \in \Lambda^2 \T_pM \text{ simple and } \norm{v} = 1
\bigr\}.
\]
Its global comass is
\[
\Comass{\omega}
:=
\sup_{p \in M} \comassp{\omega}{p}.
\]
Equivalently, for every \(p \in M\) and every simple \(v \in \Lambda^2 \T_pM\),
\[
\abs{\omega_p(v)} \le \comassp{\omega}{p}\,\norm{v}.
\]

For \(\phi \in \HZ^2(M;\R)\), its comass norm is
\[
\comassnorm{\phi}
:=
\inf \bigl\{
\Comass{\omega} : \omega \in \Dforms^2(M),\ \D\omega = 0,\ [\omega] = \phi
\bigr\}.
\]
This defines a norm on \(\HZ^2(M;\R)\). In particular, for every
\(\phi \in \HZ^2(M;\R)\) and every \(\varepsilon > 0\), there exists a
smooth closed real \(2\)-form \(\omega\) such that
\[
[\omega] = \phi
\qquad\text{and}\qquad
\Comass{\omega} \le \comassnorm{\phi} + \varepsilon.
\]

\subsection{\texorpdfstring{\(\Spinc\)}{Spin c} structures and Dirac operators}

If \(c \in \HZ^2(M;\Z)\), we write \(c_{\R} \in \HZ^2(M;\R)\) for its image under the
natural map \(\HZ^2(M;\Z) \to \HZ^2(M;\R)\).
Suppose now that \(M = M^{2m}\) is even-dimensional. A \(\Spinc\) structure on \(M\) comes with a complex spinor bundle
\[
S = S^+ \oplus S^-
\]
and a Hermitian determinant line bundle \(L \to M\) satisfying
\[
\chernclass_1(L) \equiv w_2(\T M) \mod 2.
\]
Given a unitary connection \(A\) on \(L\), we write
\[
\Dirac_A^+ \colon \Ct^\infty(M,S^+) \to \Ct^\infty(M,S^-),
\qquad
\Dirac_A \colon \Ct^\infty(M,S) \to \Ct^\infty(M,S)
\]
for the associated chiral and full \(\Spinc\) Dirac operators. Our curvature convention %
is that \(\LineCurv_A\) is an \(\iu \R\)-valued \(2\)-form; thus we may write
\[
\LineCurv_A = 2\pi \iu \, \eta
\]
for a real closed \(2\)-form \(\eta\), and then \([\eta] = \chernclass_1(L)_{\R}\).

The Atiyah--Singer index theorem~\cite[Theorem~D.15]{LawsonMichelsohn1989} gives
\begin{equation}\label{equation:spinc-index}
\ind(\Dirac_A^+)
= \pair{[M]}{\eu^{\chernclass_1(L)/2}\Ahat(\T M)},
\end{equation}
where \(\Ahat(\T M)\in \HZ^{4*}(M;\Q)\) is the \(\Ahat\)-class of \(\T M\); see \cite[Chapter III,\S 11]{LawsonMichelsohn1989}.
Moreover, the \(\Spinc\) Lichnerowicz formula~\cite[Theorem~D.12]{LawsonMichelsohn1989} reads
\begin{equation}\label{equation:spinc-lichnerowicz}
\Dirac_A^2
= \nabla_A^* \nabla_A + \frac{\scal_g}{4} + \frac{1}{2} \clm(\LineCurv_A).
\end{equation}

\section{A general comass comparison result for scalar curvature}\label{sec:abstract-results}

We start with a linear algebra lemma which relates the comass of a form to the corresponding curvature term in the Lichnerowicz formula.

\begin{lemma}\label{lem:general-clifford}
Let \(V\) be an \(n\)-dimensional Euclidean vector space, let \(E\) be a Hermitian \(\Cl(V)\)-module with Clifford multiplication \(\clm\), and set \(m:=\lfloor\tfrac{n}{2}\rfloor\).
Then for every \(\eta\in \Lambda^2V^*\),
\[
    \abs{\iu\,\clm(\eta)}_\op\le m\,\abs{\eta}_{\ast}.
\]
Moreover, if \(C>0\), \(\abs{\eta}_{\ast}\le C\), and \(0\neq \psi\in E\) satisfy \(\innp{\iu\,\clm(\eta)\psi}{\psi}=-mC\,\abs{\psi}^2\), and if \(\omega\in \Lambda^2V^*\) is defined by
\[
    \omega(v,w):=-\frac{\innp{\iu\,\clm(v\wedge w)\psi}{\psi}}{\abs{\psi}^2}
    \qquad\text{for }v,w\in V,
\]
then \(\eta=C\,\omega\). Equivalently, there exists an orthonormal basis \(e_1,\dots,e_n\) of \(V\) such that \(\omega=\sum_{j=1}^m e^{2j-1}\wedge e^{2j}\) and \(\iu\,\clm(e_{2j-1})\clm(e_{2j})\psi=-\psi\) for every \(j=1,\dots,m\). In particular, \(\iu\,\clm(\eta)\psi=-mC\,\psi\).
\end{lemma}
\begin{proof}
Choose an orthonormal basis \(e_1,\dots,e_n\) of \(V\) such that
\[
    \eta=\sum_{j=1}^m \lambda_j\, e^{2j-1}\wedge e^{2j}
\]
for suitable \(\lambda_1,\dots,\lambda_m\in \R\).
Set \(C_0:=\abs{\eta}_{\ast}=\max_{1\le j\le m}\abs{\lambda_j}\) and \(A_j:=\iu\,\clm(e_{2j-1})\clm(e_{2j})\).
Then \(\iu\,\clm(\eta)=\sum_{j=1}^m \lambda_j A_j\).
For each \(j\), the operator \(A_j\) is self-adjoint and satisfies \(A_j^2=\id\), hence has operator norm \(1\).
Therefore
\[
    \abs{\iu\,\clm(\eta)}_\op
    \le \sum_{j=1}^m \abs{\lambda_j}
    \le m \max_{1\le j\le m}\abs{\lambda_j}
    = mC_0.
\]
Now assume that \(C>0\), \(\abs{\eta}_{\ast}\le C\), and that \(0\neq \psi\in E\) satisfies
\[
    \innp{\iu\,\clm(\eta)\psi}{\psi}=-mC\,\abs{\psi}^2.
\]
Then
\[
    -mC\,\abs{\psi}^2
    =
    \sum_{j=1}^m \lambda_j\,\innp{A_j\psi}{\psi}
    \ge
    -\sum_{j=1}^m \abs{\lambda_j}\,\abs{\psi}^2
    \ge
    -mC\,\abs{\psi}^2.
\]
Hence equality holds throughout.
In particular, \(\abs{\lambda_j}=C\) for every \(j\), and \(\lambda_j\,\innp{A_j\psi}{\psi}=-\abs{\lambda_j}\,\abs{\psi}^2\) for every \(j=1,\dots,m\).
If \(\varepsilon_j:=\lambda_j/C\in\{\pm 1\}\), then \(\innp{\varepsilon_j A_j\psi}{\psi}=-\abs{\psi}^2\).
Since \(\varepsilon_j A_j\) is self-adjoint and has operator norm \(1\), this forces \(\varepsilon_j A_j\psi=-\psi\).
After replacing \(e_{2j}\) by \(-e_{2j}\) whenever \(\varepsilon_j=-1\), we may therefore assume that \(\lambda_j=C\) and \(A_j\psi=-\psi\) for every \(j\).
Define \(\omega\in \Lambda^2V^*\) by \(\omega(v,w):=-\abs{\psi}^{-2}\innp{\iu\,\clm(v\wedge w)\psi}{\psi}\).
Then \(\omega(e_{2j-1},e_{2j})=-\abs{\psi}^{-2}\innp{A_j\psi}{\psi}=1\) for every \(j\).
Finally, let \(a<b\) be a pair not equal to any of the pairs \((2j-1,2j)\), and set \(B:=\iu\,\clm(e_a)\clm(e_b)\).
Then \(B\) anticommutes with some \(A_j\), namely one for which \(\{a,b\}\) meets \(\{2j-1,2j\}\) in exactly one index.
Using that \(A_j\) is self-adjoint and \(A_j\psi=-\psi\), we obtain \(\innp{B\psi}{\psi}=\innp{A_j B\psi}{A_j\psi}=-\innp{B\psi}{\psi}\). Therefore \(\omega(e_a,e_b)=0\).
Hence \(\omega=\sum_{j=1}^m e^{2j-1}\wedge e^{2j}\), so \(\eta=C\,\omega\), and \(\iu\,\clm(\eta)\psi=\sum_{j=1}^m C A_j\psi=-mC\,\psi\). \qedhere
\end{proof}

\subsection{Even-dimensional case}
Next, we prove Theorem~\ref{thm:general_comass_comparison intro}, which we restate for convenience.

\begin{theorem}\label{thm:general_comass_comparison}
    Let \(M^{2m}\) be an even-dimensional closed oriented connected manifold and let \(c \in \HZ^2(M;\Z)\) such that \(c \equiv w_2(\T M) \mod 2\) and 
    \[\pair{[M]}{\eu^{c/2}\Ahat(\T M)} \neq 0.\]
    Let \(c_{\R}\in \HZ^2(M;\R)\) denote the image of \(c\) under the natural map \(\HZ^2(M;\Z)\to \HZ^2(M;\R)\).
    Then for every Riemannian metric \(g\) on \(M\), we have
    \[
        \min_{M} \scal_g \leq 4 m \pi \comassnorm{c_{\R}}.
    \]
    Moreover, if equality holds and \(c_{\R}\neq 0\), then \((M,g)\) is Kähler--Einstein with Kähler form \(\omega\) satisfying
    \(
        [\omega]=\frac{c_{\R}}{\comassnorm{c_{\R}}}
    \)
    in \(\HZ^2(M;\R)\).
\end{theorem}

\begin{proof}%
Fix a Riemannian metric \(g\) on \(M\), and set \(\sigma:=\min_M \scal_g\).
We will show that \(\sigma\le 4m\pi\,\comassnorm{c_{\R}}\).

Since \(c\equiv w_2(\T M)\mod 2\), there exists a \(\Spinc\) structure on \(M\) whose determinant line bundle \(L\) satisfies \(\chernclass_1(L)=c\).
Let \(S=S^+\oplus S^-\) be the associated complex spinor bundle.
Let \(c_{\R}\in \HZ^2(M;\R)\) denote the image of \(c\) under the natural map \(\HZ^2(M;\Z)\to \HZ^2(M;\R)\).

Fix \(\varepsilon>0\).
By the definition of \(\comassnorm{c_{\R}}\), there exists a smooth closed real \(2\)-form \(\eta\) with
\[
    [\eta]=c_{\R},
    \qquad
    \Comass{\eta}\le \comassnorm{c_{\R}}+\varepsilon.
\]

Choose a unitary connection \(A\) on \(L\) with curvature
\[
    \LineCurv_A=2\pi \iu\,\eta.
\]
Indeed, if \(A_0\) is any unitary connection on \(L\), then \(\frac{\LineCurv_{A_0}}{2\pi \iu}\) is a closed representative of \(\chernclass_1(L)=c\), so \(\eta-\frac{\LineCurv_{A_0}}{2\pi \iu}=\D\alpha\) for some real \(1\)-form \(\alpha\), and \(A:=A_0+2\pi \iu\,\alpha\) has the required curvature.

Let \(\Dirac_A^+:\Ct^\infty(M,S^+)\to \Ct^\infty(M,S^-)\) be the chiral \(\Spinc\) Dirac operator, and let \(\Dirac_A:\Ct^\infty(M,S)\to \Ct^\infty(M,S)\) be the full Dirac operator.
By the Atiyah--Singer index theorem,
\[
    \ind(\Dirac_A^+)
    =
    \pair{[M]}{\eu^{\chernclass_1(L)/2}\Ahat(\T M)}
    =
    \pair{[M]}{\eu^{c/2}\Ahat(\T M)}
    \neq 0.
\]
Hence \(\ker(\Dirac_A)\neq 0\).
Choose
\[
    0\neq \psi\in \Ct^\infty(M,S)
    \qquad\text{with}\qquad
    \Dirac_A\psi=0.
\]

The \(\Spinc\) Lichnerowicz formula gives
\[
    \Dirac_A^2
    =
    \nabla_A^*\nabla_A+\frac{\scal_g}{4}+\frac12\,\clm(\LineCurv_A).
\]
Integrating by parts yields
\[
    0
    =
    \norm{\nabla_A\psi}_{\Lp^2}^2
    +
    \int_M
    \left(
        \frac{\scal_g}{4}\abs{\psi}^2
        +
        \frac12\innp{\clm(\LineCurv_A)\psi}{\psi}
    \right)\,\dV_g.
\]
Since \(\LineCurv_A=2\pi \iu\,\eta\), we have \(\frac12\,\clm(\LineCurv_A)=\pi\iu\,\clm(\eta)\).
By \cref{lem:general-clifford}, applied to \(V=\T_pM\) and \(E=S_p\), at each point \(p\in M\),
\[
    \frac12\innp{\clm(\LineCurv_A)\psi}{\psi}
    =
    \pi\,\innp{\iu\,\clm(\eta)\psi}{\psi}
    \ge -\pi\,\abs{\iu\,\clm(\eta_p)}_\op\,\abs{\psi}^2
    \ge -\pi m\,\comassp{\eta}{p}\,\abs{\psi}^2.
\]
Therefore
\[
    0
    \ge
    \norm{\nabla_A\psi}_{\Lp^2}^2
    +
    \int_M
    \left(
        \frac{\scal_g}{4}
        -
        \pi m\,\Comass{\eta}
    \right)
    \abs{\psi}^2\,\dV_g.
\]
Since \(\scal_g\ge \sigma\), we obtain
\[
    0
    \ge
    \left(
        \frac{\sigma}{4}
        -
        \pi m\,\Comass{\eta}
    \right)
    \int_M \abs{\psi}^2\,\dV_g.
\]
Because \(\psi\neq 0\), the integral is positive, so
\[
    \sigma\le 4m\pi\,\Comass{\eta}
    \le 4m\pi\bigl(\comassnorm{c_{\R}}+\varepsilon\bigr).
\]
This holds for every \(\varepsilon>0\).
Letting \(\varepsilon\downarrow 0\) gives \(\sigma\le 4m\pi\,\comassnorm{c_{\R}}\).
Since \(\sigma=\min_M \scal_g\), this proves the asserted inequality.

Assume now that \(\min_M \scal_g = 4m\pi\,\comassnorm{c_{\R}}\) and \(c_{\R}\neq 0\).
Set \(C:=\comassnorm{c_{\R}}\).
Since \(c_{\R}\neq 0\) and \(\comassnorm{\cdot}\) is a norm on \(\HZ^2(M;\R)\), we have \(C>0\).

\begin{lemma}\label{lem:general-equality-parallel-spinor}
Suppose that we are in the setting of \cref{thm:general_comass_comparison} and that
\[
    \min_M \scal_g = 4m\pi\,\comassnorm{c_{\R}}.
\]
Set \(C:=\comassnorm{c_{\R}}\).
Then there exist a closed real \(\Lp^\infty\) \(2\)-form \(\eta_\infty\), a unitary connection \(A_\infty\) on \(L\) of class \(\bigcap_{p<\infty}\SobolevW^{1,p}\) with
\[
    \LineCurv_{A_\infty}=2\pi \iu\,\eta_\infty,
\]
and a nonzero spinor \(\psi_\infty\in \SobolevH^1(S)\) such that
\[
    \Comass{\eta_\infty}\le C
    \qquad\text{and}\qquad
    \nabla_{A_\infty}\psi_\infty=0.
\]
Moreover, \(\abs{\psi_\infty}\) is constant and nonzero, and
\[
    \scal_g\equiv 4m\pi C,
    \qquad
    \iu\,\clm(\eta_\infty)\psi_\infty=-mC\,\psi_\infty
\]
almost everywhere.
\end{lemma}

\begin{proof}
Let \(C:=\comassnorm{c_{\R}}\) and \(\sigma:=\min_M \scal_g = 4m\pi C\).
Let \(c_{\R}\in \HZ^2(M;\R)\) denote the image of \(c\) under the natural map \(\HZ^2(M;\Z)\to \HZ^2(M;\R)\).
Fix a smooth unitary connection \(A_0\) on \(L\) with \(\LineCurv_{A_0}=2\pi \iu\,\eta_0\) for some \(\eta_0\in \Dforms^2(M)\).
Choose a sequence \(\varepsilon_j\downarrow 0\).
For each \(j\), choose a smooth closed real \(2\)-form \(\eta_j\) with \([\eta_j]=c_{\R}\) and \(\Comass{\eta_j}\le C+\varepsilon_j\).
Since \(\eta_j\) and \(\eta_0\) represent the same class in \(\HZ^2(M;\R)\), the difference \(\eta_j-\eta_0\) is exact.
Let \(a_j\) be the unique real \(1\)-form satisfying
\[
    \D a_j=\eta_j-\eta_0,
    \qquad
    \Dadj a_j=0,
    \qquad
    a_j\perp \mathcal H^1(M),
\]
where \(\mathcal H^1(M)\) denotes the space of harmonic \(1\)-forms.
Define a unitary connection \(A_j\) on \(L\) by \(A_j=A_0+2\pi \iu\,a_j\).
Then \(\LineCurv_{A_j}=2\pi \iu\,\eta_j\).
Choose a spinor \(\psi_j\in \Ct^\infty(M,S)\) such that
\[
    \Dirac_{A_j}\psi_j=0,
    \qquad
    \norm{\psi_j}_{\Lp^2}=1.
\]
As in the proof of the inequality, we then have
\[
    0
    \ge
    \norm{\nabla_{A_j}\psi_j}_{\Lp^2}^2
    +
    \left(
        \frac{\sigma}{4}
        -
        \pi m\,\Comass{\eta_j}
    \right)
    \norm{\psi_j}_{\Lp^2}^2.
\]
Since \(\sigma=4m\pi C\) and \(\Comass{\eta_j}\le C+\varepsilon_j\), it follows that \(\norm{\nabla_{A_j}\psi_j}_{\Lp^2}^2 \le \pi m\,\varepsilon_j\to 0\).

For every \(p<\infty\), the Hodge estimate gives a constant \(K_p\) such that every \(1\)-form \(a\) orthogonal to \(\mathcal H^1(M)\) satisfies
\[
    \norm{a}_{\SobolevW^{1,p}}
    \le
    K_p\bigl(
        \norm{\D a}_{\Lp^p}
        +
        \norm{\Dadj a}_{\Lp^p}
    \bigr).
\]
Applying this to \(a_j\), we obtain \(\norm{a_j}_{\SobolevW^{1,p}} \le K_p\,\norm{\eta_j-\eta_0}_{\Lp^p}\).
Since the comass norm and any fixed bundle norm on \(\Lambda^2T^*M\) are equivalent, the bound \(\Comass{\eta_j}\le C+\varepsilon_j\) gives a uniform \(\Lp^\infty\)-bound on \(\eta_j\), hence a uniform \(\Lp^p\)-bound on \(\eta_j-\eta_0\).
Therefore \((a_j)\) is uniformly bounded in \(\SobolevW^{1,p}\) for every \(p<\infty\).
Choose \(p>2m\).
After passing to a subsequence we may assume that
\[
    a_j\rightharpoonup a_\infty \text{ weakly in } \SobolevW^{1,p}.
\]
Fix \(\alpha\in \bigl(0,1-\frac{2m}{p}\bigr)\).
The Morrey embedding
\[
    \SobolevW^{1,p}(M)\hookrightarrow \Ct^{0,\alpha}(M)
\]
is compact, so after passing to a further subsequence we also have
\[
    a_j\to a_\infty \text{ in } \Ct^{0,\alpha}
\]
after relabeling.
Set \(A_\infty:=A_0+2\pi \iu\,a_\infty\).
The uniform \(\Lp^\infty\)-bound on \(\eta_j\) implies, after passing to the same subsequence, that there exists a real \(2\)-form \(\eta_\infty\in \Lp^\infty\) such that \(\eta_j\rightharpoonup \eta_\infty\) in the weak-\(\ast\) topology on \(\Lp^\infty\), viewing \(\Lp^\infty\) as the dual space of \(\Lp^1\).
Since \(\D a_j=\eta_j-\eta_0\) and \(a_j\rightharpoonup a_\infty\) weakly in \(\SobolevW^{1,p}\), we have \(\D a_j\rightharpoonup \D a_\infty\) weakly in \(\Lp^p\), hence in the sense of distributions.
Therefore \(\eta_\infty=\eta_0+\D a_\infty\) distributionally.
In particular, \(\eta_\infty\) is closed and \(\LineCurv_{A_\infty}=2\pi \iu\,\eta_\infty\).

Because \(\nabla_{A_j}-\nabla_{A_0}\) is a zeroth-order operator whose coefficients are linear in \(a_j\), the uniform \(\Ct^0\)-bound on \(a_j\), the normalization \(\norm{\psi_j}_{\Lp^2}=1\), and the estimate above give a uniform \(\SobolevH^1\)-bound on \(\psi_j\).
Passing to a further subsequence, we may assume \(\psi_j\to \psi_\infty\) in \(\Lp^2\).
In particular, \(\norm{\psi_\infty}_{\Lp^2}=1\), so \(\psi_\infty\neq 0\).

Moreover, \(\nabla_{A_\infty}\psi_j = \nabla_{A_j}\psi_j + \bigl(\nabla_{A_\infty}-\nabla_{A_j}\bigr)\psi_j\).
The first term tends to \(0\) in \(\Lp^2\) by the almost-parallel estimate, and the second term tends to \(0\) in \(\Lp^2\) because \(A_j\to A_\infty\) in \(\Ct^0\) while \(\norm{\psi_j}_{\Lp^2}=1\).
Hence \(\nabla_{A_\infty}\psi_j\to 0\) in \(\Lp^2\).
Since \(\nabla_{A_\infty}-\nabla_{A_0}\) is a zeroth-order operator with bounded coefficients, there exists a constant \(B>0\) such that for all \(j\) and \(k\),
\[
    \norm{\nabla_{A_0}(\psi_j-\psi_k)}_{\Lp^2}
    \le
    \norm{\nabla_{A_\infty}(\psi_j-\psi_k)}_{\Lp^2}
    +
    B\,\norm{\psi_j-\psi_k}_{\Lp^2}.
\]
Hence
\[
    \norm{\nabla_{A_0}(\psi_j-\psi_k)}_{\Lp^2}
    \le
    \norm{\nabla_{A_\infty}\psi_j}_{\Lp^2}
    +
    \norm{\nabla_{A_\infty}\psi_k}_{\Lp^2}
    +
    B\,\norm{\psi_j-\psi_k}_{\Lp^2}\to 0.
\]
Therefore \((\psi_j)\) is Cauchy in \(\SobolevH^1\), so \(\psi_j\to \psi_\infty\) in \(\SobolevH^1\).
We conclude that \(\nabla_{A_\infty}\psi_\infty=0\).
Finally, since the \(\Lp^\infty\)-comass norm is the dual norm to the \(\Lp^1\)-mass norm, it is weak-* lower semicontinuous on \(\Lp^\infty\).
Hence
\[
    \Comass{\eta_\infty}\le \liminf_{j\to\infty}\Comass{\eta_j}\le C.
\]

Since \(A_\infty\) is unitary, we have \(\D\abs{\psi_\infty}^2 = 0\) distributionally.
Hence \(\abs{\psi_\infty}\) is constant.
Since \(\psi_\infty\neq 0\), this constant is positive.
After rescaling \(\psi_\infty\), we may assume that \(\abs{\psi_\infty}\equiv 1\) almost everywhere.
In particular, \(\psi_\infty\in \Lp^\infty\).
Since \(A_\infty\in \SobolevW^{1,p}\subset \Ct^{0,\alpha}\) for some \(p>2m\), the equation \(\nabla_{A_\infty}\psi_\infty=0\) shows that \(\psi_\infty\in \SobolevW^{1,\infty}\).

Since \(\nabla_{A_\infty}\psi_\infty=0\), we also have \(\Dirac_{A_\infty}\psi_\infty=0\).
Hence, exactly as above, the \(\Spinc\) Lichnerowicz formula gives
\[
    0
    =
    \norm{\nabla_{A_\infty}\psi_\infty}_{\Lp^2}^2
    +
    \int_M
    \left(
        \frac{\scal_g}{4}\abs{\psi_\infty}^2
        +
        \pi\innp{\iu\,\clm(\eta_\infty)\psi_\infty}{\psi_\infty}
    \right)\,\dV_g
\]
and therefore
\[
    0
    \ge
    \int_M
    \left(
        \frac{\scal_g}{4}
        -
        \pi m C
    \right)\abs{\psi_\infty}^2\,\dV_g
    \ge 0.
\]
Since \(\abs{\psi_\infty}\) is constant and nonzero, it follows that \(\scal_g\equiv 4m\pi C\) on \(M\).
Returning to the integral identity above, we obtain
\[
    0
    =
    \pi\int_M
    \left(
        mC\,\abs{\psi_\infty}^2
        +
        \innp{\iu\,\clm(\eta_\infty)\psi_\infty}{\psi_\infty}
    \right)\,\dV_g.
\]
The integrand is nonnegative almost everywhere by \cref{lem:general-clifford}, hence it vanishes almost everywhere.
Therefore, by the rigidity statement in \cref{lem:general-clifford},
\[
    \iu\,\clm(\eta_\infty)\psi_\infty=-mC\,\psi_\infty
\]
almost everywhere.
\end{proof}

By \cref{lem:general-equality-parallel-spinor}, there exist a closed real \(\Lp^\infty\) \(2\)-form \(\eta_\infty\), a unitary connection \(A_\infty\) on \(L\), and a nonzero spinor \(\psi_\infty\in \SobolevH^1(S)\) such that
\[
    \LineCurv_{A_\infty}=2\pi \iu\,\eta_\infty,
    \qquad
    \Comass{\eta_\infty}\le C,
    \qquad
    \nabla_{A_\infty}\psi_\infty=0.
\]
The function \(\abs{\psi_\infty}\) is constant and nonzero, and \(\iu\,\clm(\eta_\infty)\psi_\infty=-mC\,\psi_\infty\) almost everywhere.
After rescaling \(\psi_\infty\), we may assume that \(\abs{\psi_\infty}\equiv 1\) almost everywhere.
In particular, \(\psi_\infty\in \Lp^\infty\).
Since \(A_\infty\in \SobolevW^{1,p}\subset \Ct^{0,\alpha}\) for some \(p>2m\), the equation \(\nabla_{A_\infty}\psi_\infty=0\) shows that \(\psi_\infty\in \SobolevW^{1,\infty}\).

Define a real \(2\)-form \(\omega\) by
\[
    \omega(X,Y):=-\innp{\iu\,\clm(X\wedge Y)\psi_\infty}{\psi_\infty}.
\]
Then \(\omega\in \SobolevW^{1,\infty}\), and since Clifford multiplication is parallel and \(\nabla_{A_\infty}\psi_\infty=0\), we have \(\nabla\omega=0\) distributionally and hence \(\omega\) is smooth and parallel.

For almost every \(p\in M\), apply the rigidity statement in \cref{lem:general-clifford} to \(V=\T_pM\), \(E=S_p\), \(\eta=(\eta_\infty)_p\), and \(\psi=\psi_\infty(p)\).
This yields an orthonormal basis \(e_1,\dots,e_{2m}\) of \(\T_pM\) such that
\[
    (\eta_\infty)_p=C\,\omega_p=C\sum_{j=1}^m e^{2j-1}\wedge e^{2j}
\]
and, with \(A_j:=\iu\,\clm(e_{2j-1})\clm(e_{2j})\),
\[
    A_j\psi_\infty(p)=-\psi_\infty(p)
    \qquad\text{for every }j=1,\dots,m.
\]
Therefore \(\eta_\infty=C\,\omega\) almost everywhere.
Since \(\eta_\infty=\eta_0+\D a_\infty\) distributionally and \([\eta_0]=c_{\R}\), it follows that \([\eta_\infty]=c_{\R}\) in \(\HZ^2(M;\R)\).

Since \(\omega\) is smooth and parallel, the form \(\eta_\infty=C\,\omega\) is smooth.
Moreover, \(a_j\to a_\infty\) in \(\Ct^{0,\alpha}\), hence in \(\Lp^2\).
Since each \(a_j\) is orthogonal to \(\mathcal H^1(M)\), the same holds for \(a_\infty\).
Because \(\eta_\infty=\eta_0+\D a_\infty\) and \(\Dadj a_\infty=0\) distributionally, we have \((\Dadj\D+\D\Dadj)a_\infty=\Dadj(\eta_\infty-\eta_0)\) distributionally.
The right-hand side is smooth, so elliptic regularity implies that \(a_\infty\) is smooth.
Hence \(A_\infty=A_0+2\pi \iu\,a_\infty\) is a smooth unitary connection.
Since \(\nabla_{A_\infty}\psi_\infty=0\), standard regularity implies that \(\psi_\infty\) is smooth.
The identities \(\abs{\psi_\infty}\equiv 1\) and \(\iu\,\clm(\eta_\infty)\psi_\infty=-mC\,\psi_\infty\) therefore hold everywhere on \(M\).

Let \(J\) be the skew-adjoint endomorphism defined by
\[
    g(JX,Y)=\omega(X,Y).
\]
Since \(\omega\) is parallel, so is \(J\).
Fix \(p\in M\).
By \cref{lem:general-clifford}, applied to \(V=\T_pM\), \(E=S_p\), \(\eta=(\eta_\infty)_p\), and \(\psi=\psi_\infty(p)\), there exists an orthonormal basis \(e_1,\dots,e_{2m}\) of \(\T_pM\) such that
\[
    (\eta_\infty)_p=C\,\omega_p=C\sum_{j=1}^m e^{2j-1}\wedge e^{2j}
\]
and, with \(A_j:=\iu\,\clm(e_{2j-1})\clm(e_{2j})\),
\[
    A_j\psi_\infty(p)=-\psi_\infty(p)
    \qquad\text{for every }j=1,\dots,m.
\]
Hence \(Je_{2j-1}=e_{2j}\) and \(Je_{2j}=-e_{2j-1}\), so \(J^2=-\id\) at \(p\).
Moreover, \(\clm(Je_{2j-1})\psi_\infty(p)=\clm(e_{2j})\psi_\infty(p)=-\iu\,\clm(e_{2j-1})\psi_\infty(p)\), and similarly \(\clm(Je_{2j})\psi_\infty(p)=-\iu\,\clm(e_{2j})\psi_\infty(p)\).
Since \(p\) was arbitrary, we conclude that \(J^2=-\id\) on \(M\) and \(\clm(JX)\psi_\infty=-\iu\,\clm(X)\psi_\infty\) for every vector field \(X\).
Therefore \((M,g,J)\) is Kähler, with Kähler form \(\omega\).
Since \(\omega(X,Y)=g(JX,Y)\), we have \(\ins_X\omega=(JX)^\flat\), and thus
\[
    \clm(\ins_X\eta_\infty)\psi_\infty
    =
    C\,\clm(\ins_X\omega)\psi_\infty
    =
    C\,\clm(JX)\psi_\infty
    =
    -\iu\,C\,\clm(X)\psi_\infty
\]
for every vector field \(X\).

Since \(\psi_\infty\) is parallel, contracting the curvature identity (see~e.g.~\cite[Lemma~3.1]{moroianu-parallel}) gives
\[
    0
    =
    \frac12\,\clm(\Ric(X))\psi_\infty
    -
    \frac12\,\clm(\ins_X\LineCurv_{A_\infty})\psi_\infty
\]
for every vector field \(X\).
Since \(\LineCurv_{A_\infty}=2\pi\iu\,\eta_\infty\), this becomes
\[
    0
    =
    \frac12\,\clm(\Ric(X))\psi_\infty
    -
    \pi\iu\,\clm(\ins_X\eta_\infty)\psi_\infty
    =
    \frac12\,\clm(\Ric(X)-2\pi C\,X)\psi_\infty
\]
for every vector field \(X\).
Taking pointwise norms and using \(\abs{\psi_\infty}\equiv 1\), we obtain \(\abs{\Ric(X)-2\pi C\,X}=0\) for every vector field \(X\).
Hence \(\Ric(X)=2\pi C\,X\) for every \(X\).
Therefore \(\Ric=2\pi C\,g\), and \(g\) is Einstein.
Since \(\eta_\infty=C\,\omega\) and \([\eta_\infty]=c_{\R}\), we also have
\[
    [\omega]=\frac{c_{\R}}{C}=\frac{c_{\R}}{\comassnorm{c_{\R}}}
\]
in \(\HZ^2(M;\R)\).
Hence \((M,g,J)\) is Kähler--Einstein with Kähler form \(\omega\) satisfying the asserted cohomological identity.
\end{proof}

\subsection{Odd-dimensional case}
\begin{theorem}\label{thm:general_comass_comparison_odd}
    Let \(M^{2m+1}\) be an odd-dimensional closed connected oriented manifold and let \(c \in \HZ^2(M;\Z)\), \(\xi \in \HZ^1(M;\Z)\) such that \(c \equiv w_2(\T M) \mod 2\) and 
    \[\pair{[M]}{\xi \smile \eu^{c/2} \Ahat(\T M)} \neq 0.\]
    Let \(c_{\R}\in \HZ^2(M;\R)\) denote the image of \(c\) under the natural map \(\HZ^2(M;\Z)\to \HZ^2(M;\R)\).
    Then for every Riemannian metric \(g\) on \(M\), we have
    \[
        \min_{M} \scal_g \leq 4 m \pi \comassnorm{c_{\R}}.
    \]
    Moreover, if equality holds and \(c_{\R}\neq 0\), then there exists a parallel \(2\)-form \(\omega\) on \(M\) with
    \(
        [\omega]=\frac{c_{\R}}{\comassnorm{c_{\R}}}
    \),
    and the universal cover of \((M,g)\) is isometric to \(N\times \R\), where \(N\) is a Kähler--Einstein manifold of real dimension \(2m\). If \(\widetilde\omega\) denotes the lift of \(\omega\) to \(N\times \R\) and \(\omega_N\) the Kähler form of \(N\), then \(\widetilde\omega=\proj_N^*\omega_N\).
\end{theorem}

\begin{proof}
Fix a Riemannian metric \(g\) on \(M\), set \(\sigma:=\min_M \scal_g\), and fix \(\varepsilon>0\).

Since \(c\equiv w_2(\T M)\mod 2\), there exists a \(\Spinc\) structure on \(M\) whose determinant line bundle \(L\) satisfies \(\chernclass_1(L)=c\).
Let \(S\) be the associated complex spinor bundle.
Choose a smooth closed real \(2\)-form \(\eta\) with
\[
    [\eta]=c_{\R},
    \qquad
    \Comass{\eta}\le \comassnorm{c_{\R}}+\varepsilon,
\]
and choose a unitary connection \(A\) on \(L\) with
\[
    \LineCurv_A=2\pi \iu\,\eta,
\]
exactly as in the proof of \cref{thm:general_comass_comparison}.

Choose a smooth map \(u\colon M\to \Sphere^1\) representing \(\xi\), and set
\[
    \alpha:=\frac{1}{2\pi \iu}u^{-1}\D u\in \Dforms^1(M).
\]
Then \(\alpha\) is closed and \([\alpha]\) is the image of \(\xi\) in \(\HZ^1(M;\R)\).
For \(t\in[0,1]\), let
\[
    A^t:=A+4\pi \iu\, t\,\alpha
\]
be the corresponding unitary connection on \(L\).
Write \(A^\bullet:=(A^t)_{t\in[0,1]}\) and \(\Dirac_{A^\bullet}:=(\Dirac_{A^t})_{t\in[0,1]}\).
Since \(\alpha\) is closed, each \(A^t\) has the same curvature as \(A\), namely
\[
    \LineCurv_{A^t}=\LineCurv_A=2\pi \iu\,\eta.
\]
Hence the \(\Spinc\) Lichnerowicz formula is
\[
    \Dirac_{A^t}^2
    =
    \nabla_{A^t}^*\nabla_{A^t}+\frac{\scal_g}{4}+\frac12\,\clm(\LineCurv_A),
\]
exactly as in the even-dimensional case.
By \cref{lem:general-clifford}, applied to \(V=\T_pM\) and \(E=S_p\), since \(\dim M=2m+1\), for every \(p\in M\) and every \(\beta\in \Lambda^2\T_p^*M\),
\[
    \abs{\iu\,\clm(\beta)}_\op\le m\,\comassp{\beta}{p}.
\]
Thus the only new input is the production of a nonzero kernel vector.

Since \(2\pi \iu\,\alpha=u^{-1}\D u\), we have
\[
    A^1=A+2u^{-1}\D u.
\]
Thus \(A^1\) is gauge-equivalent to \(A^0=A\) via the gauge transformation \(u^2\colon L\to L\); on the spinor bundle the induced action is multiplication by \(u\), so \(\Dirac_{A^1}=u^{-1}\Dirac_{A^0}u\).
Hence, by the spectral flow index theorem (see e.g.~\textcites{Getzler:odd}[Corollary 1]{Baer-Ziemke}),
\[
    \operatorname{sf}\bigl(\Dirac_{A^\bullet}\bigr)
    = - \pair{[M]}{\xi \smile \eu^{c/2}\Ahat(\T M)}
    \neq 0,
\]
and so there exists \(t_0\in[0,1]\) and
\(
    0\neq \psi\in \Ct^\infty(M,S)\)
with
\(
    \Dirac_{A^{t_0}}\psi=0.
\)
Set \(B:=A^{t_0}\).

Applying the same Lichnerowicz/comass estimate as in the proof of \cref{thm:general_comass_comparison} to \(\psi\) yields
\[
    0
    \ge
    \norm{\nabla_B\psi}_{\Lp^2}^2
    +
    \int_M
    \left(
        \frac{\scal_g}{4}
        -
        \pi m\,\Comass{\eta}
    \right)
    \abs{\psi}^2\,\dV_g.
\]
Since \(\scal_g\ge \sigma\), we obtain
\[
    0
    \ge
    \left(
        \frac{\sigma}{4}
        -
        \pi m\,\Comass{\eta}
    \right)
    \int_M \abs{\psi}^2\,\dV_g.
\]
Because \(\psi\neq 0\), the integral is positive, so
\[
    \sigma\le 4m\pi\,\Comass{\eta}
    \le 4m\pi\bigl(\comassnorm{c_{\R}}+\varepsilon\bigr).
\]
This holds for every \(\varepsilon>0\).
Letting \(\varepsilon\downarrow 0\) gives
\(
    \min_M \scal_g \le 4m\pi\,\comassnorm{c_{\R}},
\)
as desired.

Assume now that \(\min_M \scal_g = 4m\pi\,\comassnorm{c_{\R}}\) and \(c_{\R}\neq 0\).
Set \(C:=\comassnorm{c_{\R}}\).
Then \(C>0\) and \(\sigma=4m\pi C\).

\begin{lemma}\label{lem:general-equality-parallel-spinor-odd}
Suppose that we are in the setting of \cref{thm:general_comass_comparison_odd} and that
\[
    \min_M \scal_g = 4m\pi\,\comassnorm{c_{\R}}.
\]
Set \(C:=\comassnorm{c_{\R}}\).
Then there exist a closed real \(\Lp^\infty\) \(2\)-form \(\eta_\infty\), a unitary connection \(A_\infty\) on \(L\) of class \(\bigcap_{p<\infty}\SobolevW^{1,p}\) with
\[
    \LineCurv_{A_\infty}=2\pi \iu\,\eta_\infty,
\]
and a nonzero spinor \(\psi_\infty\in \SobolevH^1(S)\) such that
\[
    \Comass{\eta_\infty}\le C
    \qquad\text{and}\qquad
    \nabla_{A_\infty}\psi_\infty=0.
\]
Moreover, \(\abs{\psi_\infty}\) is constant and nonzero, and
\[
    \scal_g\equiv 4m\pi C,
    \qquad
    \iu\,\clm(\eta_\infty)\psi_\infty=-mC\,\psi_\infty
\]
almost everywhere.
\end{lemma}

\begin{proof}
Fix a smooth unitary connection \(A_0\) on \(L\) with \(\LineCurv_{A_0}=2\pi \iu\,\eta_0\) for some \(\eta_0\in \Dforms^2(M)\).
Choose \(\varepsilon_j\downarrow 0\).
For each \(j\), choose a smooth closed real \(2\)-form \(\eta_j\) with \([\eta_j]=c_{\R}\) and \(\Comass{\eta_j}\le C+\varepsilon_j\).
Define \(A_j\) from \(\eta_j\) exactly as in the proof of \cref{lem:general-equality-parallel-spinor}, so that \(\LineCurv_{A_j}=2\pi \iu\,\eta_j\).

Fix \(u\colon M\to \Sphere^1\) representing \(\xi\), and define \(\alpha\) as above.
For \(t\in[0,1]\), set
\[
    A_j^t:=A_j+4\pi \iu\, t\,\alpha.
\]
Write \(A_j^\bullet:=(A_j^t)_{t\in[0,1]}\) and \(\Dirac_{A_j^\bullet}:=(\Dirac_{A_j^t})_{t\in[0,1]}\).
As before,
\[
    \operatorname{sf}\bigl(\Dirac_{A_j^\bullet}\bigr)
    =
    - \pair{[M]}{\xi \smile \eu^{c/2}\Ahat(\T M)}
    \neq 0,
\]
so there exist \(t_j\in[0,1]\) and \(\psi_j\in \Ct^\infty(M,S)\) such that
\[
    \Dirac_{A_j^{t_j}}\psi_j=0,
    \qquad
    \norm{\psi_j}_{\Lp^2}=1.
\]
Set \(B_j:=A_j^{t_j}\).
The Lichnerowicz estimate from the first part gives
\[
    0
    \ge
    \norm{\nabla_{B_j}\psi_j}_{\Lp^2}^2
    +
    \left(
        \frac{\sigma}{4}
        -
        \pi m\,\Comass{\eta_j}
    \right)
    \norm{\psi_j}_{\Lp^2}^2.
\]
Since \(\sigma=4m\pi C\) and \(\Comass{\eta_j}\le C+\varepsilon_j\), it follows that \(\norm{\nabla_{B_j}\psi_j}_{\Lp^2}^2\le \pi m\,\varepsilon_j\to 0\).

As in the proof of \cref{lem:general-equality-parallel-spinor}, the bound \(\Comass{\eta_j}\le C+\varepsilon_j\) yields uniform \(\SobolevW^{1,p}\)-bounds on \(A_j-A_0\) for every \(p<\infty\).
Since
\[
    B_j-A_0=(A_j-A_0)+4\pi \iu\, t_j\,\alpha
\]
and \((t_j)\subset[0,1]\), the sequence \((B_j-A_0)\) is also uniformly bounded in \(\SobolevW^{1,p}\) for every \(p<\infty\).
Applying exactly the same compactness argument as in the proof of \cref{lem:general-equality-parallel-spinor} to the sequence \((B_j,\psi_j)\), and choosing \(p>2m+1\), we then obtain, after passing to a subsequence, a connection \(A_\infty\) on \(L\), a closed real \(\Lp^\infty\) \(2\)-form \(\eta_\infty\), and a nonzero spinor \(\psi_\infty\in \SobolevH^1(S)\) such that
\[
    \LineCurv_{A_\infty}=2\pi \iu\,\eta_\infty,
    \qquad
    \Comass{\eta_\infty}\le C,
    \qquad
    \nabla_{A_\infty}\psi_\infty=0.
\]
Moreover, the same equality-case argument as in \cref{lem:general-equality-parallel-spinor} implies that \(\abs{\psi_\infty}\) is constant and nonzero, that \(\scal_g\equiv 4m\pi C\), and that \(\iu\,\clm(\eta_\infty)\psi_\infty=-mC\,\psi_\infty\) almost everywhere. \qedhere
\end{proof}

By \cref{lem:general-equality-parallel-spinor-odd}, there exist a unitary connection \(A_\infty\) on \(L\), a closed real \(\Lp^\infty\) \(2\)-form \(\eta_\infty\), and a nonzero spinor \(\psi_\infty\in \SobolevH^1(S)\) such that \(\nabla_{A_\infty}\psi_\infty=0\) and \(\iu\,\clm(\eta_\infty)\psi_\infty=-mC\,\psi_\infty\) almost everywhere.
After rescaling \(\psi_\infty\), we may assume \(\abs{\psi_\infty}\equiv 1\).

Define a real \(2\)-form \(\omega\) by
\[
    \omega(X,Y):=-\innp{\iu\,\clm(X\wedge Y)\psi_\infty}{\psi_\infty}.
\]
Since Clifford multiplication is parallel and \(\nabla_{A_\infty}\psi_\infty=0\), we have \(\nabla\omega=0\) distributionally, hence \(\omega\) is smooth and parallel.

For almost every \(p\in M\), apply the rigidity statement in \cref{lem:general-clifford} to \(V=\T_pM\), \(E=S_p\), \(\eta=(\eta_\infty)_p\), and \(\psi=\psi_\infty(p)\).
This yields an orthonormal basis \(e_1,\dots,e_{2m+1}\) of \(\T_pM\) such that
\[
    (\eta_\infty)_p=C\,\omega_p=C\sum_{j=1}^m e^{2j-1}\wedge e^{2j}
\]
and, with \(B_j:=\iu\,\clm(e_{2j-1})\clm(e_{2j})\),
\[
    B_j\psi_\infty(p)=-\psi_\infty(p)
    \qquad\text{for every }j=1,\dots,m.
\]
Hence
\[
    \ker(\omega_p)=\R e_{2m+1}.
\]
Since \(\omega\) is parallel, \(\ker(\omega)\subset \T M\) is a parallel \(1\)-dimensional distribution.

Since \(\omega\) is smooth and \(\eta_\infty=C\,\omega\) almost everywhere, the form \(\eta_\infty\) is smooth.
Exactly as in the even-dimensional equality case, elliptic regularity then implies that \(A_\infty\) is smooth, and standard regularity implies that \(\psi_\infty\) is smooth.
Hence all identities above hold everywhere on \(M\).
Exactly as in the even-dimensional equality case we obtain \([\eta_\infty]=c_{\R}\) in \(\HZ^2(M;\R)\).
Since \(\eta_\infty=C\,\omega\), we have
\[
    [\omega]=\frac{c_{\R}}{C}=\frac{c_{\R}}{\comassnorm{c_{\R}}}
\]
in \(\HZ^2(M;\R)\).

Let \(\pi\colon (\widetilde{M},\tilde g)\to (M,g)\) be the universal cover.
Since \(M\) is closed, \((\widetilde{M},\tilde g)\) is complete and simply connected.
Pull back \(A_\infty\), \(\eta_\infty\), \(\omega\), and \(\psi_\infty\) to \(\widetilde{M}\), and keep the same notation.
The form \(\omega\) is parallel with \(1\)-dimensional kernel, so by the de~Rham splitting theorem there is an isometric splitting
\[
    (\widetilde{M},\tilde g)\cong (N,g_N)\times \R.
\]
Moreover, \(\omega\) restricts to a parallel nondegenerate \(2\)-form \(\omega_N\) on \(N\).
Let \(J\in \End(\T N)\) be defined by
\[
    g_N(JX,Y)=\omega_N(X,Y).
\]
By construction, \(\omega=\proj_N^*\omega_N\) under this splitting.
The pointwise normal form above shows that \(J^2=-\id\), so \(J\) is a parallel complex structure and \((N,g_N,J)\) is Kähler.

Finally, since \(\psi_\infty\) is parallel, the curvature identity \cite[Lemma~3.1]{moroianu-parallel} gives
\[
    0
    =
    \frac12\,\clm(\Ric(X))\psi_\infty
    -
    \frac12\,\clm(\ins_X\LineCurv_{A_\infty})\psi_\infty
\]
for every vector field \(X\).
Since \(\LineCurv_{A_\infty}=2\pi \iu\,\eta_\infty=2\pi \iu\,C\,\omega\), this becomes
\[
    0
    =
    \frac12\,\clm(\Ric(X))\psi_\infty
    -
    \pi\iu\,\clm(\ins_X\eta_\infty)\psi_\infty.
\]
If \(X\in \ker(\omega)\), then \(\ins_X\eta_\infty=0\), hence \(\clm(\Ric(X))\psi_\infty=0\) and thus \(\Ric(X)=0\).
If \(X\perp \ker(\omega)\), then \(\ins_X\omega=(JX)^\flat\) and \(\clm(JX)\psi_\infty=-\iu\,\clm(X)\psi_\infty\) as in the even-dimensional case, so \(\clm(\ins_X\eta_\infty)\psi_\infty=-\iu\,C\,\clm(X)\psi_\infty\) and therefore \(\Ric(X)=2\pi C\,X\).
In particular, \((N,g_N)\) is Einstein.
Since it is Kähler, it is Kähler--Einstein, and the claimed splitting of \((\widetilde{M},\tilde g)\) follows. \qedhere
\end{proof}

\section{Geometric Applications}\label{sec:applications}

\subsection{Stable norm, duality and calibrations}
Recall from \cref{ss:comass} that \(\Comass{\cdot}\) denotes the comass of a \(2\)-form and
\(\comassnorm{\cdot}\) the induced norm on \(\HZ^2(M;\R)\).
We now recall the dual homological notion that will be used in the applications.

A \(2\)-current on \(M\) is a continuous linear functional
\[
T:\Dforms^2(M)\to \R,
\]
where \(\Dforms^2(M)\) is equipped with its standard \(\Ct^\infty\)-topology.
Its boundary is the \(1\)-current \(\bd T\) defined by
\[
(\bd T)(\alpha):=T(\D\alpha), \qquad \alpha\in \Dforms^1(M).
\]
A \(2\)-current is a cycle if \(\bd T=0\). The mass of a \(2\)-current \(T\) is%
\[
M(T):=\sup\{\abs{T(\omega)}: \omega\in \Dforms^2(M),\ \Comass{\omega}\le 1\}.
\]
Let \(Z_2(M;\R)\) denote the vector space of \(2\)-cycles of finite mass and
let \(B_2(M;\R)\subset Z_2(M;\R)\) be the subspace of boundaries.
Via the standard identification
\[
\HZ_2(M;\R)\cong Z_2(M;\R)/B_2(M;\R),
\]
the stable norm of \(h\in \HZ_2(M;\R)\) is
\[
\stnorm{h}:=\inf\{M(T): T\in Z_2(M;\R),\ [T]=h\}.
\]
If \(\HZ_2(M;\Z)/\torsion\neq 0\), the stable \(2\)-systole is
\[
\sys^\st_2(M,g)
:=\min\{\stnorm{h}: 0\neq h\in \HZ_2(M;\Z)/\torsion\}.
\]

The following classical theorem states the duality between the stable norm on homology and the comass norm on cohomology; see Federer \cite[Theorem 4.10]{Federer1975} and
Gromov \cite[Proposition 4.35]{gromov-metric-structures}.

\begin{theorem}\label{theorem:stable-comass-duality}
For every \(h \in \HZ_2(M;\R)\),
\[
\stnorm{h}
= \sup\bigl\{ \pair{\varphi}{h} : \varphi \in \HZ^2(M;\R),\ \comassnorm{\varphi} \le 1 \bigr\}.
\]
\end{theorem}

When \(\HZ_2(M;\Z)\) has rank one, this duality takes a particularly nice form.

\begin{corollary}\label{corollary:rank-one-duality}
Suppose that \(\HZ_2(M;\Z)/\torsion \cong \Z\). Let \(a\) be a generator, and
let \(x \in \HZ^2(M;\Z)\) be the dual generator, that is, \(\pair{x}{a} = 1\). Then
\[
\comassnorm{x_\R} = \frac{1}{\stnorm{a}}.
\]
\end{corollary}

\begin{proof}
Since \(\HZ^2(M;\R)\) is one-dimensional, every class \(\varphi \in \HZ^2(M;\R)\) can be
written as \(\varphi = t x_\R\) for some \(t \in \R\). Then
\[
\pair{\varphi}{a} = t
\qquad\text{and}\qquad
\comassnorm{\varphi} = \abs{t}\, \comassnorm{x_\R}.
\]
By \cref{theorem:stable-comass-duality},
\[
\stnorm{a}
= \sup\bigl\{ t : t \in \R,\ \abs{t}\, \comassnorm{x_\R} \le 1 \bigr\}
= \frac{1}{\comassnorm{x_\R}}.\qedhere
\]
\end{proof}

A closed \(2\)-form \(\omega\in \Dforms^2(M)\) with \(\Comass{\omega}\le 1\) is called a calibration.
If \(T\in Z_2(M;\R)\) and \(\omega\) is a calibration such that \(T(\omega)=M(T)\), then \(T\) is mass minimizing in its real homology class. Indeed,
\[
M(T)=T(\omega)=\pair{[\omega]}{[T]} \le \stnorm{[T]} \le M(T),
\]
so equality holds throughout.

In particular, if \(\Sigma\subset M\) is a compact oriented closed surface calibrated
by \(\omega\), i.e.
\[
\omega|_\Sigma=\vol_\Sigma,
\]
then its integration current \(T_\Sigma\) satisfies
\[
M(T_\Sigma)=\Area_g(\Sigma)=\int_\Sigma \omega,
\]
and hence
\[
\stnorm{[\Sigma]}=\Area_g(\Sigma).
\]
On a Kähler manifold \((M,g,J)\), Wirtinger's inequality implies that the Kähler form has comass \(1\), and every complex
curve with its complex orientation is calibrated by it; see \cite[Example I, p.~59]{HL-calibrated-geometries}.
For the calibration formalism and the mass-minimizing property of calibrated currents, we refer the reader to the foundational paper by Harvey and Lawson \cite{HL-calibrated-geometries}.

\subsection{Complex projective spaces}\label{sec:cpn}
Let \(J_0\) denote the standard complex structure on \(\CP^n\).
We denote by \(\mathcal O_{\CP^n}(-1)\) the tautological line bundle and by \(\mathcal O_{\CP^n}(1)\) its dual bundle.
More generally, we consider powers \(\mathcal O_{\CP^n}(k)=\mathcal O_{\CP^n}(1)^{\otimes k}\).
Let
\[
h:=\chernclass_1\bigl(\mathcal O_{\CP^n}(1)\bigr)\in \HZ^2(\CP^n;\Z).
\]
Then \(\HZ^2(\CP^n;\Z)\cong \Z\) is generated by \(h\). If
\(\ell\cong \CP^1\subset \CP^n\) is a projective line, endowed with its complex
orientation, then its homology class
\[
a:=[\ell]\in \HZ_2(\CP^n;\Z)
\]
generates \(\HZ_2(\CP^n;\Z)\cong \Z\), and the pairing satisfies
\[
\pair{h}{a}=\int_\ell h = 1.
\]
Equivalently, \(h\) is the Poincaré dual of the class of a hyperplane
\(\CP^{n-1}\subset \CP^n\).
Moreover, \((\CP^n,J_0)\) is Fano of index \(n+1\), since
\[
K^{-1}_{\CP^n}\cong \mathcal O_{\CP^n}(n+1),
\]
where \(K^{-1}_{\CP^n}\) denotes the anticanonical bundle, that is, the dual of the canonical bundle \(K_{\CP^n}=\Lambda^{n,0}\T^*\CP^n\).

Let \(g_\FS\) be the Fubini--Study metric on \((\CP^n,J_0)\), normalized by
\[
\int_\ell \omega_\FS=\pi
\]
for a projective line \(\ell\subset \CP^n\). Then
\[
[\omega_\FS]=\pi h.
\]
Moreover, \(g_\FS\) is Kähler--Einstein with
\[
\Ric_{g_\FS}=2(n+1)\,g_\FS,
\qquad
\scal_{g_\FS}=4n(n+1).
\]

Next, we prove Theorem \ref{thm:main} which we restate below for convenience.

\begin{theorem}\label{thm:cpn-rigidity}
Let \(M\) be a smooth manifold diffeomorphic to \(\CP^n\).
Let \(g\) be a Riemannian metric on \(M\) such that
\[
\scal_g \ge 4n(n+1).
\]
Then
\[
\sys_2^\st(M,g)\le \pi.
\]
Moreover, if equality holds, then there exist a complex structure \(J\) on \(M\) and
a biholomorphism
\[
B \colon (M,J)\to (\CP^n,J_0)
\]
such that
\[
g=B^*g_\FS.
\]
\end{theorem}

\begin{proof}
Fix a diffeomorphism \(f\colon M\to \CP^n\) and endow \(M\) with the pullback orientation.
Let
\[
b:=f_*^{-1}a\in \HZ_2(M;\Z),\qquad x:=f^*h\in \HZ^2(M;\Z).
\]
Then \(\pair{x}{b}=\pair{h}{a}=1\).
Every nonzero class in \(\HZ_2(M;\Z)\) is of the form \(kb\) with \(k\in \Z\setminus\{0\}\).
By homogeneity,
\[
\sys^\st_2(M,g)=\stnorm{b}.
\]

Set
\[
c:=(n+1)x\in \HZ^2(M;\Z),
\]
and let \(c_{\R}\in \HZ^2(M;\R)\) be its image in real cohomology.
As \(f\) is a diffeomorphism, \(c\equiv w_2(\T M)\pmod 2\) and \(\Ahat(\T M)=f^*\Ahat(\T\CP^n)\).
Moreover,
\[
\eu^{c/2}\Ahat(\T M)
=
f^*\!\left(\eu^{\chernclass_1(\T\CP^n)/2}\Ahat(\T\CP^n)\right)
=
f^*\!\Td(\T\CP^n)
\]
from which
\[
\pair{[M]}{\eu^{c/2}\Ahat(\T M)}
=
\pair{[\CP^n]}{\Td(\T\CP^n)}
=
\chi(\mathcal O_{\CP^n})
=
1.
\]
Thus the hypotheses of \cref{thm:general_comass_comparison} are satisfied for the class \(c\).
Applying \cref{thm:general_comass_comparison} with \(m=n\) and using \cref{corollary:rank-one-duality}, we obtain
\[
4n(n+1)
\le
\min_M \scal_g
\le
4n\pi \comassnorm{c_{\R}}
=
4n\pi(n+1)\comassnorm{x_\R}
=
\frac{4n(n+1)\pi}{\sys^\st_2(M,g)}.
\]
Hence
\[
\sys_2^\st(M,g)\le \pi.
\]

Assume now that \(\sys^\st_2(M,g)=\pi\).
Then equality holds throughout the previous chain of inequalities, so equality holds in \cref{thm:general_comass_comparison}.
It follows that there exists a complex structure \(J\) on \(M\) such that \((M,g,J)\) is Kähler--Einstein with Einstein constant \(2(n+1)\), and its Kähler form \(\omega\) satisfies
\[
[\omega]
=
\frac{c_{\R}}{\comassnorm{c_{\R}}}
=
\frac{(n+1)x_\R}{(n+1)\comassnorm{x_\R}}
=
\pi x_\R.
\]
Since \(M\) is diffeomorphic to \(\CP^n\), by the Hirzebruch--Kodaira--Yau theorem (see \cite[Theorem~1.1]{Tosatti2017} originally proved in \cite{HirzebruchKodaira1957,Yau1977}), there is a biholomorphism
\[
F \colon (M,J) \to (\CP^n,J_0).
\]

If we set
\[
\widetilde{\omega} := (F^{-1})^* \omega,
\qquad
\widetilde{g} := (F^{-1})^* g,
\]
then \(\widetilde{g}\) is a Kähler--Einstein metric on the standard complex manifold \(\CP^n\). Let \(\widetilde{\rho}\) denote its Ricci form.
Since \(F\) is a biholomorphism,
\[
\widetilde{\rho} = (F^{-1})^* \rho = 2(n+1)\widetilde{\omega}.
\]
Using Chern--Weil theory for \(\T^{1,0}\CP^n\) (see \cite[\S4.4]{HuybrechtsCG}), together with \cite[Prop.~4.A.11]{HuybrechtsCG} and the previous identity, we conclude
\begin{equation}\label{eq:c_1=(n+1)x}
(n+1)h=\chernclass_1(\T^{1,0}\CP^n)=\frac{\iu}{2\pi}[\tr_{\C}(F_\nabla)]=\frac{[\tilde\rho]}{2\pi}=\frac{(n+1)[\tilde\omega]}{\pi}.
\end{equation}
Hence
\[
[\widetilde{\omega}] = \pi h = [\omega_\FS].
\]

Thus \(\widetilde{g}\) is a Kähler--Einstein metric on the standard complex manifold
\(\CP^n\), and its Kähler form \(\widetilde{\omega}\) lies in the same Kähler
class as \(\omega_\FS\). In order to apply the Bando--Mabuchi uniqueness theorem~\cite[Theorem~A]{BandoMabuchi1987},
we rescale and set
\[
\widehat{\omega} := 2(n+1)\widetilde{\omega},
\qquad
\widehat{\omega}_\FS := 2(n+1)\omega_\FS.
\]
Since constant rescaling does not change the Ricci form, both
\(\widehat{\omega}\) and \(\widehat{\omega}_\FS\) are Kähler--Einstein and satisfy
\[
[\widehat{\omega}] = [\widehat{\omega}_\FS] = 2\pi \chernclass_1(\CP^n).
\]
By the Bando--Mabuchi uniqueness theorem, there exists a holomorphic automorphism \(\Phi \in \Aut(\CP^n)\) such that \(\widehat{\omega} = \Phi^* \widehat{\omega}_\FS\).
Equivalently,
\[
\widetilde{\omega} = \Phi^* \omega_\FS,
\qquad
\widetilde{g} = \Phi^* g_\FS.
\]
Therefore
\[
g = F^* \widetilde{g} = F^*(\Phi^* g_\FS) = (\Phi \circ F)^* g_\FS.
\]
Setting
\[
B := \Phi \circ F \colon (M,J) \to (\CP^n,J_0)
\]
completes the proof.
\end{proof}

\begin{remark}
The constant \(\pi\) in \cref{thm:cpn-rigidity} is sharp. Indeed, for the normalized Fubini--Study metric \(g_\FS\), the K\"ahler form \(\omega_\FS\) is a calibration and calibrates every projective line \(\ell \cong \CP^1\), see \cite[Example~1, page 59]{HL-calibrated-geometries}.
Hence \(\ell\) is mass minimizing in its real homology class and therefore
\[
\stnorm{a}=\Area_{g_\FS}(\ell)=\int_\ell \omega_\FS=\pi.
\]
\end{remark}

\subsection{Extremality and rigidity in the odd-dimensional case}
We now prove Theorem \ref{thm:main-odd}, which we restate below for convenience.

\begin{theorem}\label{thm:cP^1xS^1}
Let \(M\) be a smooth manifold diffeomorphic to \(\CP^n \times \Sphere^1\).
Let \(g\) be a Riemannian metric on \(M\) such that
\[
\scal_g \ge 4n(n+1).
\]
Then
\[
\sys^\st_2(M,g)\le \pi.
\]
Moreover, if equality holds, then the universal cover of \((M,g)\) is isometric to
\[
(\CP^n,g_\FS)\times \R.
\]
\end{theorem}

\begin{proof}
Fix a diffeomorphism \(f\colon M\to \CP^n\times \Sphere^1\) and endow \(M\) with the pullback orientation.
Let \(p_1,p_2\) denote the two projections from \(\CP^n\times \Sphere^1\) onto \(\CP^n\) and \(\Sphere^1\), respectively.
Retain the notation \(a\in \HZ_2(\CP^n;\Z)\) and \(h\in \HZ^2(\CP^n;\Z)\) from \cref{sec:cpn}, where \(a=[\ell]\) is the class of a projective line and \(h=\chernclass_1(\mathcal O_{\CP^n}(1))\).
Choose \(s_0\in \Sphere^1\) and set
\[
b:=f^{-1}_*[\ell\times\{s_0\}]\in \HZ_2(M;\Z),\qquad
x:=f^*p_1^*h\in \HZ^2(M;\Z),
\]
and
\[
\xi:=f^*p_2^*\theta\in \HZ^1(M;\Z),
\]
where \(\theta\in \HZ^1(\Sphere^1;\Z)\) denotes the positive generator. Then
\[
\pair{x}{b} =1.
\]
Since \(\HZ_2(M;\Z)\cong \Z\), every nonzero class in \(\HZ_2(M;\Z)\) is of the form
\(kb\) with \(k\in \Z\setminus\{0\}\). By homogeneity,
\[
\sys^\st_2(M,g)=\stnorm{b}.
\]

Set
\[
c:=(n+1)x\in \HZ^2(M;\Z),
\]
and let \(c_{\R}\in \HZ^2(M;\R)\) be its image in real cohomology. Since \(f\) is a
diffeomorphism,
\[
c\equiv w_2(\T M)\pmod 2.
\]
Moreover,
\[
\eu^{c/2}\Ahat(\T M)
=
f^*\bigl(p_1^*\Td(\T\CP^n)\bigr),
\]
and hence
\[
\begin{aligned}
\pair{[M]}{\xi\smile \eu^{c/2}\Ahat(\T M)}
&=
\pair{[\CP^n\times \Sphere^1]}{p_2^*\theta\smile p_1^*\Td(\T\CP^n)}\\
&=
\pair{[\CP^n]}{\Td(\T\CP^n)}
=
\chi(\mathcal O_{\CP^n})=1.
\end{aligned}
\]
Thus the hypotheses of \cref{thm:general_comass_comparison_odd} are satisfied for the classes \(c\) and \(\xi\).

Applying \cref{thm:general_comass_comparison_odd} with \(m=n\) and using \cref{corollary:rank-one-duality}, we obtain
\[
4n(n+1)
\le \min_M \scal_g
\le 4n\pi \comassnorm{c_{\R}}
=4n\pi(n+1)\comassnorm{x_{\R}}
=
\frac{4n(n+1)\pi}{\sys^\st_2(M,g)}.
\]
Hence
\[
\sys^\st_2(M,g)\le \pi.
\]

Assume now that \(\sys^\st_2(M,g)=\pi\).
Then equality holds throughout the previous chain of inequalities, so equality holds in \cref{thm:general_comass_comparison_odd}.
It follows that there exists a parallel \(2\)-form \(\omega\) on \(M\) such that
\[
[\omega]
=
\frac{c_{\R}}{\comassnorm{c_{\R}}}
=
\pi x_\R,
\]
and that the universal cover \((\widetilde M,\widetilde g)\) is isometric to
\[
(N,g_N)\times \R,
\]
where \(N\) is a Kähler--Einstein manifold of real dimension \(2n\).
If \(\widetilde\omega\) denotes the lift of \(\omega\) to \(\widetilde M\), and \(\omega_N\) the Kähler form of \(N\), then
\[
\widetilde\omega=\proj_N^*\omega_N.
\]

Let \(q:\widetilde M\to M\) be the universal covering.
Since \(\widetilde M\) is diffeomorphic to \(\CP^n\times\R\), the manifold \(N\) is closed and homotopy equivalent to \(\CP^n\).
Hence, \(\HZ^2(\widetilde M;\Z)\cong \HZ^2(N;\Z)\cong \Z\).
Therefore, there exists a unique class \(h_N\in \HZ^2(N;\Z)\) such that
\[
q^*x=\proj_N^*h_N.
\]
Since \([\widetilde\omega]=\pi q^*x\), we obtain
\[
[\omega_N]=\pi h_N.
\]

Furthermore, the scalar curvature of the product \((N,g_N)\times\R\) is pulled back from \(N\), and the covering map \(q:(\widetilde M,\widetilde g)\to(M,g)\) is a local isometry.
Therefore,
\[
\proj_N^*\scal_{g_N}=
\scal_{\widetilde g}=
q^*\scal_g
\equiv 4n(n+1).
\]
Thus,
\[
\scal_{g_N}\equiv 4n(n+1).
\]
Since \(g_N\) is Kähler--Einstein of complex dimension \(n\), its Ricci form
\(\rho_N\) satisfies
\[
\rho_N = 2(n+1)\omega_N.
\]
Therefore
\[
\chernclass_1(N)=\frac{[\rho_N]}{2\pi}
=(n+1)\frac{[\omega_N]}{\pi}
=(n+1)h_N.
\]
Thus \(N\) is Fano of index \(n+1\).
By the theorem of Kobayashi--Ochiai~\cite[Theorem A]{DedieuHoering2017} (or originally~\cite{KobayashiOchiai1973}), there exists a biholomorphism
\[
F:(N,J)\to (\CP^n,J_0).
\]
If we set
\[
\widetilde\omega_N:=(F^{-1})^*\omega_N,\qquad
\widetilde g_N:=(F^{-1})^*g_N,
\]
then \(\widetilde g_N\) is a Kähler--Einstein metric on \(\CP^n\) and
\[
[\widetilde\omega_N]=\pi h=[\omega_\FS].
\]
Arguing exactly as in the proof of \cref{thm:cpn-rigidity} and applying the Bando--Mabuchi uniqueness theorem~\cite[Theorem~A]{BandoMabuchi1987}, we conclude that there exists a biholomorphism
\[
B:(N,J)\to (\CP^n,J_0)
\]
such that
\[
g_N=B^*g_\FS.
\]
Hence
\[
(\widetilde M,\widetilde g)\cong (\CP^n,g_\FS)\times \R.
\]
This completes the proof.
\end{proof}

\begin{remark}
The constant \(\pi\) in \cref{thm:cP^1xS^1} is sharp.
Indeed, for the product metric \(g_\FS+(\DD t)^2\) on \(\CP^n\times \Sphere^1\), the closed \(2\)-form \(\proj_1^*\omega_\FS\) is a calibration and calibrates every projective line \(\ell\times\{s\}\).
Hence
\[
\sys^\st_2(\CP^n\times \Sphere^1,g_\FS+(\DD t)^2)=\pi.
\]
\end{remark}

\printbibliography

\end{document}